\documentclass[a4, 11pt]{amsart}
\usepackage{amssymb}
\usepackage{amstext}
\usepackage{amsmath}
\usepackage{amscd}
\usepackage{latexsym}
\usepackage{amsfonts}
\usepackage{color}
\usepackage{enumerate}
\usepackage{textcomp}
\usepackage[all]{xy}
\usepackage[latin1]{inputenc}
\theoremstyle{plain}
\newtheorem*{thm*}{Theorem}
\newtheorem*{cor*}{Corollary}
\newtheorem*{defn*}{Definition}
\newtheorem*{claim*}{Claim}

\newtheorem{theorem}{Theorem}[section]
\newtheorem{corollary}[theorem]{Corollary}
\newtheorem{lemma}[theorem]{Lemma}
\newtheorem{proposition}[theorem]{Proposition}
\newtheorem{example}[theorem]{Example}
\newtheorem{asumption}[theorem]{Assumption}
\newtheorem{definition}[theorem]{Definition}
\newtheorem{remark}[theorem]{Remark}

\theoremstyle{definition}

\theoremstyle{remark}

\def\Hom{\mathrm{Hom}}

\def\Ext{\mathrm{Ext}}

\def\End{\mathrm{End}}

\DeclareMathOperator{\reg}{reg}

\def\m{\mathfrak m}

\def\card{\mathrm{Card}}
\def\depth{\mathrm{depth}}
\def\Supp{\mathrm{Supp}}

\def\m{\mathfrak m}

\def\fz{\mathbb{Z}}
\def\fn{\mathbb{N}}

\begin{document}

\setlength{\baselineskip}{14.5pt}
\title{Castelnuovo-Mumford regularity  and Segre-Veronese transform}
\author{ Marcel
Morales}
\address{Universit\'e de Grenoble I, Institut Fourier, 
UMR 5582, B.P.74,
38402 Saint-Martin D'H\`eres Cedex,
and IUFM de Lyon, 5 rue Anselme, 69317 Lyon Cedex (FRANCE)}

\email{ morales@ujf-grenoble.fr}
\author{Nguyen Thi Dung}
\address{Thai Nguyen University of Agriculture and Forestry, Thai Nguyen, Vietnam}
\email{xsdung050764@gmail.com}

\thanks{  This work was supported partially  by VIASM, Hanoi, Vietnam.}

\begin{abstract}
In this paper we give a nice formula for the  Castelnuovo-Mumford regularity of the Segre product of modules, under some suitable hypotheses.
This extends recent results   of David A. Cox, and Evgeny Materov (2009).  

\end{abstract}

\maketitle
\tableofcontents

 \section{Introduction}Segre and Veronese embeddings of projective variety plays a key role in Algebraic Geometry.
 From the algebraic point of view $K$-algebras and their
 free resolutions are important fields of research. One of the important numerical invariants of $S$-modules is the Castelnuovo-Mumford regularity. 
Several approaches to study the
  Castelnuovo-Mumford regularity of Segre Veronese embeddings was given. The first one was done in  \cite{bm}. 
S.Goto and K. Watanabe have studied the local cohomology modules of Veronese and Segre transform of graded modules. 
Motivated by the paper of David A. Cox, and Evgeny Materov (2009), 
where is computed the Castelnuovo-Mumford regularity of the Segre Veronese embedding,
we  extend  their result and compute the  Castelnuovo-Mumford regularity of the Segre product of modules under
 an hypothesis on the persistence of the local cohomology modules; note that this hypothesis is true for Cohen-Macaulay modules. 
 In  \cite{gw} it is given a formula for the local cohomology of a Segre product of modules   with $\depth \geq 2.$ We can easily extends it to modules 
  with $\depth \geq 1.$
By definition
the Castelnuovo-Mumford regularity is the maximum of the "real" regularities of some Segre product of local cohomology modules. In section 2, we define 
the "virtual" regularities and we prove that even if 
the "real" regularities  and the "virtual" regularities are different, 
taking the maximum over the set of all "virtual" regularities gives the Castelnuovo-Mumford regularity. This allows us to give nice formulas for 
Castelnuovo-Mumford regularity of the Segre product of Cohen-Macaulay modules. 
Some cases of the above results were obtained by studying Hilbert-Poincaré series in \cite{md0}.
In section 3, we study when our hypothesis are true for Segre products of Stanley Reisner rings.

\section{Segre transform, Local cohomology and Castelnuovo-Mumford regularity }

Let $S$ be a polynomial ring over a field $K$, in a finite number of variables. We suppose that $S$ is graded by the standard graduation $S=\oplus_{i\geqslant 0}S_i$.
 Let $\m=\oplus_{i\geqslant 1}S_i$ be the maximal irrelevant ideal. Let $M$ be a finitely generated graded $S$-module, $M=\oplus_{l\geqslant \sigma}M_l,$ with $\sigma\in \fz$ and $M_{\sigma}\ne 0.$ 
\begin{remark}We choose to take as base ring   polynomials rings, because of the paper \cite{gw}, but in fact by using \cite{bs},
 all our results will be true over standard Noetherian graded rings $R= R_0[R_1]$, where $R_0$ is a local ring with infinite residue field. 
We can assume without loss of generality that the field $K$ is infinite.
 
 \end{remark}
Let $M$ be a finitely generated graded $S$-module, the local cohomology modules are  graded, so we can define 
 $\End(H^i_{\m}(M))=\max\{ \beta\in \fz \mid (H^i_{\m}(M))_{\beta}\ne 0\},$ $ r_j(M )=\End (H^j_{\m }(M ))+j $ and
 $ \reg(M )=\max r_j(M )$ the Castelnuovo-Mumford regularity of $M$.


Let $S_1, S_2$ be two polynomial rings on two disjoint sets of variables, for $i=1,2$, $M_i$ be  a graded $S_i$-module. The Segre product
 $M_1{\underline{\otimes}} M_2$ is defined by $\oplus _{n\in \fz} (M_1)_n\otimes (M_2)_n$. Note that $M_1{\underline{\otimes}} M_2$ is 
a $S_1{\underline{\otimes}} S_2$-module. 
 By using Küneth formula for global cohomology (see \cite{gw}[Proposition (4.1.5)and Remark 4.1.7],  \cite{sv}[Section 0.2]),
    we can extend \cite{gw}[Proposition (4.1.5)]:
\begin{theorem}\label{t11}
Let $S_1, S_2$ be two polynomial rings on two disjoint sets of variables, for $i=1,2$, $M_i$ be a finitely generated  graded $S_i$-module. 
Let $\m$ be the  maximal irrelevant ideal of $S_1{\underline{\otimes}} S_2$. Assume that $\dim M_1\geq 1,\dim M_2\geq 1$, and
  $\depth M_i\geqslant \min(2,\dim M_i)$ for  $i=1,2.$ Then for all integers $j,k,l\geq 1$
$$ H^j_{\m}(M_1{\underline{\otimes}} M_2)\simeq (M_1{\underline{\otimes}} H^j_{\m_2}(M_2))\oplus (H^j_{\m_1}(M_1){\underline{\otimes}} M_2)
\underset{{k,l\ {\rm s.t.} \  j=k+l-1}}\oplus (H^k_{\m_1}(M_1){\underline{\otimes}} H^l_{\m_2}(M_2)). $$

\end{theorem}
\begin{proof} The case $\dim M_1,\dim M_2\geq 2$ is \cite{gw}[Proposition (4.1.5)]. Suppose that $M_1$ is  Cohen-Macaulay of dimension 1
 and $\depth M_2\geqslant 1$.
 From \cite{gw}[Remark 4.1.7] we get two exact sequences (for the notations we refer to \cite{gw}):
 $$0 \longrightarrow  M_1\longrightarrow M_1^0\longrightarrow H^1_{\m_1}(M_1)\longrightarrow, 0$$
 $$0\longrightarrow H^1_{\m_1}(M_1){\underline{\otimes}} M_2\longrightarrow  H^1_{\m}(M_1{\underline{\otimes}} M_2)\longrightarrow 
 M_1^0{\underline{\otimes}} H^1_{\m_2}(M_2)\longrightarrow,  0$$
 from the first one we get the exact sequence:
 $$0 \longrightarrow  M_1{\underline{\otimes}} H^1_{\m_2}(M_2)\longrightarrow M_1^0{\underline{\otimes}} H^1_{\m_2}(M_2)
\longrightarrow H^1_{\m_1}(M_1){\underline{\otimes}} H^1_{\m_2}(M_2)\longrightarrow, 0$$
Our claim follows from these two exact sequences.
\end{proof} 
Note that if  $M_1, M_2$ are Cohen-Macaulay modules of dimension 1, then $M_1{\underline{\otimes}} M_2$ is a Cohen-Macaulay module of dimension 1, and 
$ \reg(M_1{\underline{\otimes}} M_2 )=\max (\reg(M_1),\reg(M_2 ))$.

\begin{definition}\label{d1} From now on, we consider $s$-polynomial rings (with disjoint set of variables) $S_1,\ldots, S_s$, graded, with irrelevant ideals $\m_1,\ldots, \m_s$. For $i=1,\ldots s,$ let $M_i$ be 
a finitely generated  graded $S_i$-module such that $\dim M_i\geq 1$ and $\depth M_i\geqslant \min(2,\dim M_i).$ 
(Without loss of generality we can assume that at most one module has dimension one).
We set the following notations:
\begin{itemize}\item $ d_i=\dim M_i;\ \ \sigma_i=\min\{l\in \fz: (M_i)_l\ne 0\};$\\  $r_{i,j}:=\End(H^j_{\m_i}(M_i))+j, $
\item $ C:=\{0,\ldots,d_1\}\times \{0,\ldots,d_2\}\times\ldots\times \{0,\ldots,d_s\}, $
\item for $u\in  C, $  $ \Supp\ u=\{i\in \{1,\ldots,s\}\mid u_i\ne 0 \} ,$ 
\item  $\displaystyle E_{i,j}=\begin{cases}H^j_{\m_i}(M_i) \ \text{\ if\ } j>0\\ M_i\ \ \ \ \ \text{\ if\ } j=0\end{cases}$;\\
 for  $u\in  C,$ $E_u=\underset{i=1}{\overset{s}{\underline{\otimes}} }E_{i,u_i}, $ 
$\tilde E_u=\underset{i\in \Supp u}{\underline{\otimes} }E_{i,u_i}. $
\end{itemize}
\end{definition} 
We can state the following corollary of Theorem \ref{t11}.
\begin{corollary} \label{c12} With the  notations introduced in \ref{d1},  for $j\geqslant 1$, we have
$$ H^j_{\m}(M_1{\underline{\otimes}}\ldots{\underline{\otimes}} M_s)\simeq \underset{u\in  C,s.t.\\  {\displaystyle \sum_{l\in \Supp u}}(u_l-1)=j-1}{\oplus} E_u, $$
where   $\m$ is the irrelevant maximal ideal of $S_1{\underline{\otimes}} \ldots{\underline{\otimes}} S_s.$ 

Hence we have that 
$$\reg(M_1{\underline{\otimes}}\ldots{\underline{\otimes}} M_s)= \underset{ u\in C}\max   \{ 1+\sum_{l\in \Supp u}(u_l-1)+\End(E_u) \}.$$
\end{corollary}
In order to compute $\reg(M_1{\underline{\otimes}}\ldots{\underline{\otimes}} M_s),$ we  have to know 
   when $E_u\not = 0$, and in this case find $\End(E_u)$. In a concrete example, if we know all local cohomology modules, it is possible 
to compute the regularity of the Segre product, but to give a formula in general is   impossible. 
Our purpose is to give a formula in terms of the regularity data $r_{i,u_i}$,  looking for the optimal hypothesis.

We say that a graded $S$-module $N=\oplus_{l\in \fz}N_l,$ not necessarily finitely generated has no gaps if there is no integers $i<j<k$ such that
$$ N_i\ne 0; N_k\ne 0; N_j=0. $$
\begin{example} \label{e9} Let $M$ be a finitely generated graded $S$-module with $\depth M\geqslant 1$, $M=\oplus_{l\geqslant \sigma}M_l,$ 
where $\sigma\in \fz$ and $M_{\sigma}\ne 0.$ 
Since the field $K$ is infinite, 
there exists  $x\in S_1$, a nonzero divisor of $M$. The multiplication  by $x$ defines an injective map $M_i\rightarrow M_{i+1}$, 
hence $M$ has no gaps, and for all $l\geq \sigma $, we have  $M_{l}\ne 0.$
\end{example}\begin{asumption}\label{assumption}{\bf From now on, we assume : for any $i=1,...,s$; 
 $\dim M_i\geq 1$ and $\depth M_i\geq \min (2,\dim M_i)$. For any $ \depth M_i\leq j\leq \dim M_i,$ if $H^{j}_{\m_i}(M_i)\not= 0$ then 
   $H^{j}_{\m_i}(M_i)$ has no gaps, 
and $(H^{j}_{\m_i}(M_i))_k\not=0$ for infinitely many $k$. (Without loss of generality we can assume that at most one module has dimension one).
}
\end{asumption}
Note that our assumption is true if all the modules $M_i$ are Cohen-Macaulay.
We introduce another piece of notation:
\begin{itemize}
\item   For $u\in C,$ $\Gamma_u=1+ \underset {l\in \Supp u}{\sum} (u_l-1)+\End(E_u). $
\item   For $u\in C,$  $ \displaystyle \gamma_u=
1+\underset{l\in \Supp u}{\sum}(u_l-1)+\underset{i\in \Supp u}{\min} \bigl(\End(H^{u_i}_{\m_i}(M_i))\bigl),$
\item   $C_1:= \{  u\in C;  E_u\not = 0 \},$ $C_2:= \{  u\in C;  \tilde E_u\not = 0 \},$

\end{itemize}
The following Lemma follows immediately from the definitions and the Assumption \ref{assumption}. 
\begin{lemma}\label{l13bis} With the  notations introduced in \ref{d1}, and the Assumption \ref{assumption}. Let $\epsilon_1,..., \epsilon_s$ be the canonical basis of $\fz^s$. For any $u\in C$:
\begin{enumerate}
\item $\tilde E_u\not = 0$ if and only if $H^{u_i}_{\m_i}(M_i)\not=0$ for all $i\in \Supp u$.
\item $ \End\tilde E_u=\underset{i\in \Supp u}{\min} (\End(H^{u_i}_{\m_i}(M_i)). $
\item  If $E_u\ne 0$  then $\End E_u=\End\tilde E_u$, that is $\Gamma _u =\gamma _u  .$  If $u$ has full support then $E_u\ne 0$.
\item For any $k\notin \Supp u,$ $\lambda_k\in \fn^*, \lambda_k\leq d_k$, such that $H^{\lambda_k}_{\m_k}(M_k)\not=0$, we have 
$$\gamma_{u+\lambda_k\epsilon_k}=\min(\gamma_u+\lambda_k-1, \End(H^{\lambda_k}_{\m_k}(M_k))+\lambda_k+\underset{l\in \Supp u}{\sum}(u_l-1)).  $$
\end{enumerate}\end{lemma}
We can state our main result:
\begin{theorem}\label{main1}  With the  notations introduced in \ref{d1}, We have:
$$\reg(M_1{\underline{\otimes}}\ldots{\underline{\otimes}} M_s)\leq \underset{u\in C_2}{\max}
 \{ 1+\underset{l\in \Supp u}{\sum}(u_l-1)+\underset{i\in \Supp u}{\min} \bigl(\End(H^{u_i}_{\m_i}(M_i))\bigl)  \}.$$
Moreover the Assumption \ref{assumption} implies the equality. 
 \end{theorem}
\begin{proof} Note that 
$$\reg(M_1{\underline{\otimes}}\ldots{\underline{\otimes}} M_s)= \max{(\Gamma_u | E_u\not=0 )}\leq \max{(\gamma_u | \tilde E_u\not=0 )}, $$
Hence the  inequality is trivial. The equality  will follows if we prove that 
$$ \max{(\Gamma_u | E_u\not=0 )}=\max{(\gamma_u | u\in C_2)}. $$

By the above Lemma  if $E_u\not=0$ then $\Gamma _u=\gamma _u$. We suppose that $E_u=0$. For any $i=1,\ldots,s$ let $\delta_i$ 
be an integer such that $\reg(M_i)=\End(H^{\delta_i}_{\m_i}(M_i))+\delta_i.$ We will prove the following statement
$$ (*)\ { \rm There\ exists \ }
n\notin \Supp u, { \rm  such\ that  \ }
 \gamma_u\leqslant \gamma_{u+\delta_n\epsilon_n}.$$ If $E_{u+\delta_n\epsilon_n}\not=0$ then 
$\gamma_u\leqslant \gamma_{u+\delta_n\epsilon_n}=\Gamma_{u+\delta_n\epsilon_n}$, otherwise we repeat the argument. 
This process ends  since  
the local cohomology module 
$E_{u+\underset{l\not\in \Supp u}{\sum}\delta_l\epsilon_l}$ is non zero by the Assumption \ref{assumption}. 
Hence there exist some set $J\subset \{1,...,n\}\setminus \Supp u $ such that $E_{u+\sum_{n\in J}\delta_n\epsilon_n}\not=0$ and 
$\gamma_u\leqslant \gamma_{u+\sum_{n\in J}\delta_n\epsilon_n}=\Gamma_{u+\sum_{n\in J}\delta_n\epsilon_n}$
the claim is true.

Now we prove $(*)$: we have 
$E_u=0 \Leftrightarrow \End(\underset{i\in\Supp u}{{\underline{\otimes}}} H^{v_i}_{\m_i}(M_i))<\underset{j\not\in \Supp u}{\max}\sigma_j.$
Let $n\notin \Supp u$ such that $\underset{j\in \Supp u}{\max}\sigma_j=\sigma_n.$ Thus the condition $E_u=0$ is equivalent to
 \hfill \break $\underset{i\in \Supp u}{\min}(\End(H^{u_i}_{\m_i}(M_i)))<\sigma_n$.
 But $\sigma_n\leqslant \reg(M_n)=\End(H^{\delta_n}_{\m_n}(M_n))+\delta_n $ by \cite[Theorem 15.3.1]{bs}. So 
$$ \underset{i\in \Supp u}{\min}(\End(H^{u_i}_{\m_i}(M_i)))\leqslant \End(H^{\delta_n}_{\m_n}(M_n))+\delta_n-1.$$ 
 It implies that 
$$ \gamma_u= 1+\underset{l\in \Supp u}{\sum}(u_l-1)+\underset{i\in \Supp u}{\min} \bigl(\End(H^{u_i}_{\m_i}(M_i))\bigl)
\leqslant \End(H^{\delta_n}_{\m_n}(M_n))+\delta_n+
\underset{l\in \Supp u}{\sum}(u_l-1). $$ 
 On the other hand, since $\depth(M_n)\geqslant 1,$ then  $\delta_n\geqslant 1,$ and  
we have trivially that  $\gamma_u\leqslant \gamma_u+(\delta_n-1)$ which implies  that
$$ \gamma_u\leqslant \min(\gamma_u+\delta_n-1, \End(H^{\delta_n}_{\m_n}(M_n))+\delta_n+\underset{l\in \Supp u}{\sum}(u_l-1))=\gamma_{u+\delta_n\epsilon_n}.$$ 
The last equality follows from  Lemma \ref{l13bis}.
\end{proof} 

In order to apply our results to Segre Veronese transform, we  recall the following Proposition from \cite{md}:
\begin{proposition}If $M$ is a  Cohen-Macaulay module of dimension $d$, then $\reg M[\tau ]^{<n>}= d-\lceil \frac{d -\reg M+\tau }{n}\rceil.$
\end{proposition}
Hence we have the following consequence:
\begin{theorem} \label{main2} Let $S_1,\ldots, S_s$ be graded polynomial rings on disjoints of set of variables. 
For all $i=1,\ldots,s,$ let $M_i$ be a graded  finitely generated $S_i$-Cohen-Macaulay module with $\dim M_i\geq 1$.  
(Without loss of generality we can assume that at most one module has dimension one). 
 Let $d_i=\dim M_i, b_i=d_i-1\geq 1$, $\alpha_i=d_i-\reg(M_i),$ where $\reg(M_i)$ is the Castelnuovo-Mumford regularity of $M_i.$ Then 

$(1)$ $\reg(M_1{\underline{\otimes}} \ldots {\underline{\otimes}} M_s)=\underset{u\in C_2}{\max} \{ 1+\underset{  l\in \Supp u}{\sum}b_l  -\underset{  l\in \Supp u}{\max}\{\alpha_l\}\}.$

$(2)$ For $n_i\in \fn,$ let $M_i[\tau_i ]^{<n_i>}$ be the shifted $n_i$-Veronese transform of $M_i$, then 
$$\reg (M_1[\tau_1 ]^{<n_1>}{\underline{\otimes}} \ldots {\underline{\otimes}} M_s[\tau_s ]^{<n_s>})=
\underset{u\in C_2}{\max} \{ 1+\underset{  l\in \Supp u}{\sum}b_l  -\underset{  l\in \Supp u}{\max}\{\lceil\frac{\alpha_l+\tau_l}{n_l}\rceil\}\}.
$$
\end{theorem}
As a Corollary we generalize one of the  main results  of \cite{cm}[Theorem 1.4]
\begin{corollary} \label{main-cm} (\cite{cm}[Theorem 1.4])For   $i=1,\ldots,s,$ let  $S_i$ be  graded polynomial rings on disjoints sets of variables,
 with $\dim S_i\geq 1. $
(Without loss of generality we can assume that at most one ring has dimension one). 
 let $m_i,n_i\in \fz$, and  $S_i[m_i]^{<n_i>}$ be the $n_i$-Veronese transform of $S_i[m_i],$  
then 
$$\reg (S_1[m_1]^{<n_1>}{\underline{\otimes}} \ldots {\underline{\otimes}} S_s[m_s]^{<n_s>})=\underset{u\in C_2}{\max} \{ 1+\underset{  l\in \Supp u}{\sum}b_l 
 -\underset{  l\in \Supp u}{\max}\{\lceil\frac{b_l+m_l+1}{n_l}\rceil\}\}.
$$
\end{corollary}

\subsection{Local cohomology Modules without gaps}
Let $M$ be a finitely generated graded $S$-module. We recall the local duality's theorem (see \cite{s}) :\hfill\break
We have an isomorphism : $$H^i_{\m}(M) \simeq \Hom_{S}(\Ext^{ n-i }_{S}(M,S), E(S/\m)).$$
We denote by $D^i (M)$ the finitely generated graded $S$-module $\Ext^{ n-i }_{S}(M,S)$. 
The following Lemma follows immediately from the local duality's theorem and the Example \ref{e9}.
\begin{lemma} \label{e9bis} \begin{itemize}\item If $\depth (D^i (M))\geq 1$ then $H^i_{\m}(M)$ has no gaps and 
and for all $l\leq  \End(H^i_{\m}(M)) $, we have   $(H^i_{\m}(M))_l\not=0$.
\item Let $M$ be a finitely generated  graded $S$-module with $\depth M\geqslant 1$ and $\dim M=d$. It is known that $\depth D^d(M)\geq \min \{ d,2\}$. Hence 
 the top local cohomology $H^d_{\m}(M)$ has no gaps 
and for all $l\leq  \End(H^d_{\m}(M)) $, we have   $(H^d_{\m}(M))_l\not=0$.
\item It follows from \cite{m-c} that if $A$ is a standard graded simplicial toric ring of dimension $d$, and $\depth A= d-1$, then $D^{d-1} (A)$ 
is a Cohen-Macaulay Module of dimension $d-2$, so if $d\geq 3$, the module $H^{d-1}_{\m}(A)$ has no gaps, and 
for all $l\leq  \End(H^{d-1}_{\m}(M))  $, we have   $(H^{d-1}_{\m}(M))_l\not=0$.
\end{itemize}\end{lemma}
\section{Square free monomial ideals}

Let $\Delta $ be a simplicial complex with support on $n$ vertices, labeled by the set $[n]=\{1,...,n\}$, $S:=K[x_1,...,x_n]$ be a polynomial ring,
 $I_\Delta \subset S$  be the Stanley Reisner ideal associated to $\Delta $, that is $I_\Delta=(x^F/ F\not\in \Delta )$. The Alexander dual of $\Delta$
 is the simplicial complex 
$\Delta^*$ defined by: $F\subset [n]$ is a face of  $\Delta^*$ if and only if $[n]\setminus F\notin \Delta $.
 The following theorem is a well known consequence of Hochster's Theorems (see for example \cite{sb}[Proposition 3.8]):
\begin{proposition} Let $K[\Delta]:= S/I_\Delta $ be the Stanley-Reisner ring associated to $\Delta $, $a=(a_1,...,a_n)\in \fz$, where for all $i$, $a_i\leq 0$.
 Let $F=\Supp(a)\subset [n]$  and $\mid F\mid:= \card(F)$ then:
 $$\dim_K (H^i_\m(K[\Delta]))_a=\beta _{i+1-\mid F\mid,[n]\setminus F }(K[\Delta^*]).$$ 
\end{proposition}
We get the following Corollary: 
\begin{corollary} \label{sbarraloccoh} \cite{sb}[Lemma 3.9]
\begin{enumerate}
\item $$\dim_K (H^i_\m(K[\Delta]))_0=\beta _{i+1,n}(K[\Delta^*]),$$ 
\item For any integer $j> 0$:
$$\dim_K (H^i_\m(K[\Delta]))_{-j}=\sum_{h=1} ^{\min(j,n)} {{n}\choose{h }} {{h+j -1}\choose{j }}\beta _{i+1-h,n-h }(K[\Delta^*]).$$ 
\end{enumerate}
\end{corollary}

 \begin{corollary}\label{sbarraloccoh1} Let $i$ be an integer such that $(H^i_{\m}(K[\Delta] ))\not=0$, define $k_i$ as the  smallest $0\leq h\leq n$ such that $\beta _{i+1-h,n-h }(K[\Delta^*])\not= 0$. 
 If $k_i\geq 1$  then $H^i_{\m}(K[\Delta] )$ has no gaps and $(H^i_{\m}(K[\Delta] ))_k\not=0$ for all  $k\leq  {\rm end}(H^i_{\m}(K[\Delta] ))$.
 \end{corollary}
  
 \begin{proof}   The formula proved in \ref{sbarraloccoh}
 implies that $(H^i_{\m}(K[\Delta] ))_k\not=0$ for all  $k\leq  k_i$. The claim is over.
 \end{proof}
Let remark that $\depth (K[\Delta])\geq 1$ always, so there is a big class of Stanley-Reisner rings that satisfies the Assumption \ref{assumption}.
 In any case if we have $s$ Stanley-Reisner rings, the Castelnuovo-Mumford regularity of their Segre product, can be read off from their Betti's tables, 
by the Corollary \ref{c12}.

\end{document}